\documentclass[11pt,letterpaper]{amsart}
\usepackage{mathrsfs,color,balance,bm,amsmath,amsfonts,amssymb,dsfont,amscd,extarrows,enumerate,verbatim}
\usepackage[all]{xy}
\numberwithin{equation}{section}

\newtheorem{thm}{Theorem}[section]
\newtheorem{cor}[thm]{Corollary}
\newtheorem{lem}[thm]{Lemma}

\newtheorem{defn}[thm]{Definition}

\newcommand{\mr}{\mathbb{R}}

\newcommand{\mc}{\mathbb{C}}

\newcommand{\rw}{\rightarrow}

\DeclareMathOperator{\loc}{loc}

\DeclareMathOperator{\diam}{diam}
\DeclareMathOperator{\dom}{Dom}

\usepackage[pagebackref=true,colorlinks=true,bookmarksopen,bookmarksnumbered,citecolor=red, linkcolor=blue, urlcolor=cyan]{hyperref}

\allowdisplaybreaks[4]
\begin{document}

\title[$L^2$-estimates on flat vector bundles and Pr\'ekopa's theorem]
{$L^2$-estimates on flat vector bundles and Pr\'ekopa's theorem}

\author[G. Huang]{Gang Huang }
\address{Gang Huang: \ School of Mathematical Sciences, University of Chinese Academy of Sciences \\ Beijing 100049, P. R. China}
\email{huanggang21@mails.ucas.ac.cn}

\author[W. Jiang]{Weiwen Jiang}
\address{Weiwen Jiang: \ Institute of Mathematics, Academy of Mathematics and Systems Science, Chinese Academy of Sciences, Beijing 100190, P. R. China}
\email{jiangweiwen@amss.ac.cn}

\author[X. Qin]{Xiangsen Qin}
\address{Xiangsen Qin: \ School of Mathematical Sciences, University of Chinese Academy of Sciences \\ Beijing 100049, P. R. China}
\email{qinxiangsen19@mails.ucas.ac.cn}

\begin{abstract}
    In this paper, we will construct H\"ormander's $L^2$-estimate of the operator $d$ on a flat vector bundle over a $p$-convex Riemannian manifold and 
    discuss some geometric applications of it. In particular, we will generalize the classical Pr\'ekopa's theorem in convex analysis.
\end{abstract}

\maketitle
\tableofcontents
\section{Introduction} 

 \par In complex analysis, the $L^2$-estimate of $\bar\partial$-equation on pseudoconvex domains was established by H\"ormander in 1965 in his fundamental work \cite{H65}. H\"ormander's result was generalized to a complete K\"ahler manifold by Demailly and other authors in the early eighties (see e.g. \cite[Theorem 4.5, Chapter VIII]{D12}),  where the metric can be not complete if one only consider $(n,q)$-forms (\cite[Theorem 6.1, Chapter VIII]{D12}).
        Surprisingly, one had not seen analogous work about the $L^2$-estimate for the $d$-operator
        until a decade later when Brascamp and Lieb proved a parallel result in \cite{BL76} on $\mr^n$.
        The $L^2$-estimate for the $d$-operator on general convex domains was proved only recently in \cite{JLY14} for trivial line bundles.
        For positively curved vector bundles on the whole $\mr^n$,
        the $L^2$-estimate for the $d$-operator parallel to Brascamp-Lieb's result is given in \cite{C19} in 2019, and then
       be generalized to a general convex domain of $\mr^n$ in \cite{Z23} and \cite{DHJQ24}.\\
      \indent As indicated in the abstract, in the present paper, we will establish  the $L^2$-estimate for $d$ on a flat vector bundle over a $p$-convex Riemannian manifold
      and discuss some geometric applications of it. Indeed, the first main result of this paper is the following:
\begin{thm}\label{thm:flat vector bundle}
Let $M$ be an  $n$-dimensional $p$-convex Riemannian manifold without boundary for some  $p\in\{1,\cdots,n\}$,
$(E,h)$ be a flat vector bundle over $M$. Suppose $\mathfrak{Ric}_p+\Theta^{(E,h)}\geq 0,$ then for any 
$d$-closed $f\in L^2_{\loc}(M,\Lambda^pT^*M\otimes E)$ satisfying 
$$\int_{M}\left\langle \left(\mathfrak{Ric}_p+\Theta^{(E,h)}\right)^{-1}f,f\right\rangle dV<\infty,$$
there exists $u\in L^2(M,\Lambda^pT^*M\otimes E)$ such that 
$$du=f\text{ and }\int_{M}|u|^2dV\leq \int_{M}\left\langle \left(\mathfrak{Ric}_p+\Theta^{(E,h)}\right)^{-1}f,f\right\rangle dV.$$
Moreover, if $f$ is smooth, then $u$ can be taken to be smooth.
\end{thm}
\indent Please see Section \ref{sec:notations} for various notations and conventions.\\
\indent Theorem \ref{thm:flat vector bundle} generalizes all known results (as mentioned above) of the $L^2$-estimate for $d$. The definitions of $d$, $\Theta^{(E,h)}$ and the proof of Theorem \ref{thm:flat vector bundle} will be given in Section \ref{sec:flat}.
A key tool is the construction of the Bochner-type identity (see Lemma \ref{lem:bochner}).\\
\indent Furthermore, when $\Omega\subset\subset M$ is an open subset which has a strictly $p$-convex boundary,
then we may derive the following improved $L^2$-estimate.
\begin{thm}\label{thm:improved $L^2$-estimates}
Let $M$ be a Riemannian manifold without boundary of dimension $n\geq 2$, $\Omega\subset\subset M$ be an open subset with 
a smooth strictly $p$-convex boundary defining function $\rho$
such that $\mathfrak{Ric}_p\geq 0$ on $\Omega$, and let $(E,h)$ be a trivial vector bundle over $M$
such that $\Theta^{(E,h)}\geq 0$ on $\Omega$. Then 
there is a constant $\delta:=\delta(\Omega,\rho,h)>0$ such that for any nonzero 
$d$-closed $f\in L^2(\Omega,\Lambda^pT^*M\otimes E)$ satisfying 
$$N_f:=\int_{\Omega}\left\langle \left(\mathfrak{Ric}_p+\Theta^{(E,h)}\right)^{-1}f,f\right\rangle dV<\infty,$$
there exists $u\in L^2(\Omega,\Lambda^pT^*M\otimes E)$ such that $du=f$ and 
$$\int_{\Omega}|u|^2dV\leq \frac{\|f\|_{L^2(\Omega)}}{\sqrt{\|f\|^2_{L^2(\Omega)}+\delta N_f}}\int_{\Omega}\left\langle \left(\mathfrak{Ric}_p+\Theta^{(E,h)}\right)^{-1}f,f\right\rangle dV,$$
where $\|f\|_{L^2(\Omega)}^2:=\int_{\Omega}|f|^2dV$. Moreover, if $f$ is smooth, then $u$ can be taken to be smooth.
\end{thm}
 \indent 
 Note that the assumption $E$ is trivial is essential in our proof for Theorem \ref{thm:improved $L^2$-estimates}.
  When $M\neq \mr^n$, our proof of Theorem \ref{thm:improved $L^2$-estimates} strongly relies on the uniform estimates of the Green operator for differential forms which is proved  in \cite{DHQ25}. When $M=\mr^n$ ($n\geq 3$) and $h$ is a constant metric, the constant $\delta$ can be taken to be a  constant only depends on $n,|\Omega|$, $h$, and the smallest eigenvalue of the Hessian of $\rho$ on $\partial\Omega$.\\ 
 \indent Similar as \cite[Theorem 1.1]{DNWZ23}, we will construct the corresponding $L^2$-inverse of Theorem \ref{thm:flat vector bundle}.
 For simplicity of statements, we introduce the following notations: we say $M$ is a Riemannian manifold without boundary satisfying condition $\clubsuit$
 if one of the following holds:
 \begin{itemize}
   \item[(i)] $M$ is $p$-convex, and for any $x\in M$, there is a strictly $p$-plurisubharm-
   onic function $f\in C^2(M)$ such that $f(y)=d(x,y)^2$ for all $y$ in an open neighborhood of 
   $x$, where $1\leq p\leq \dim(M)$.
   \item[(ii)] $n:=\dim(M)\geq 2$ and $M$ has a strictly $p$-plurisubharmonic function $\eta\in C^2(M)$, where $p$ satisfies 
   $$\begin{cases}
     p=n, & \mbox{if } n=2, \\
     3\leq p\leq n, & \mbox{if }n\geq 3.
   \end{cases}$$
 \end{itemize}
By the proof of \cite[Lemma 12.15]{L18}, we know any $p$-convex open subset of a Cartan-Hadamard manifold satisfying condition $\clubsuit$.
\begin{thm}\label{thm:L^2 inverse}
 Let $M$ be a Riemannian manifold without boundary satisfying condition $\clubsuit$, and let $(E,h)$ be a flat vector bundle over $M$, which satisfies the following: for any strictly $p$-plurisubharmonic function $\psi\in C^2(M)$, and any $d$-closed $f\in  C^\infty_c(M,\Lambda^pT^*M\otimes E),$ 
there exists a Lebesgue measurable differential form $u$ such that 
$$du=f\text{ and }\int_{M}|u|^2e^{-\psi}dV\leq \int_{M}\langle F_\psi^{-1}f,f\rangle e^{-\psi}dV.$$
Then we have $\mathfrak{Ric}_p+\Theta^{(E,h)}\geq 0$.
\end{thm}
\indent When $p=1$, $M$ is a convex domain of $\mr^n$ and $E$ is trivial, Theorem \ref{thm:L^2 inverse}
is given as \cite[Theorem 1.3]{DZ21} 
We remark that the assumption $M$ satisfying condition $\clubsuit$ is important for our proof (to construct a good $\psi$).
 With the same method used in the proof of Theorem \ref{thm:L^2 inverse}, one can conclude that the constant $\delta$ in  Theorem \ref{thm:improved $L^2$-estimates}
cannot be chosen to be a constant which is independent of the metric $h$ when $M$ satisfying condition $\clubsuit$.\\
\indent For convenience, we introduce the following definition, which is similar to the $p$-optimal $L^2$-estimate of $\bar\partial$ that is introduced in \cite{DNWZ23}.
\begin{defn}
 Let $M$ be a Riemannian manifold without boundary, $E$ be a flat vector bundle over $M$, and let $h$ be a (possibly singular) Riemannian metric on 
 $E$. Then $(M,E,h)$ is said to satisfy the $p$-optimal $L^2$-estimate condition if it satisfies the following:  for any strictly $p$-plurisubharmonic function $\psi\in C^2(M)$, and any $d$-closed $f\in  C^\infty_c(M,\Lambda^pT^*M\otimes E),$ 
there exists a Lebesgue measurable differential form $u$ such that 
$$du=f\text{ and }\int_{M}h(u,u)e^{-\psi}dV\leq \int_{M}h(F_\psi^{-1}f,f)e^{-\psi}dV,$$
where we have extended the metric $h$ to the tensor bundle $\Lambda^pT^*M\otimes E$, and which is still denoted by 
 $h.$
\end{defn} 
\indent By Theorem \ref{thm:L^2 inverse}, the proof of Theorem \ref{thm:flat vector bundle} and use Equality (\ref{metric}), we have the following:
\begin{cor}\label{cor:equivalent}
Let $M$ be a Riemannian manifold without boundary satisfying condition $\clubsuit$, 
and let $(E,h)$ be a flat vector bundle over $M$. Then $(M,E,h)$ satisfies the $p$-optimal $L^2$-estimate condition if and only  if 
$\mathfrak{Ric}_p+\Theta^{(E,h)}\geq 0.$
\end{cor}
\indent In \cite{L17}, Lempert asked whether a Hermitian metric whose curvature dominates zero is Nakano semi-positive.
This question is answered affirmatively by Liu-Xiao-Yang-Zhou in \cite{LXYZ24} by using the $p$-optimal $L^2$-estimate of $\bar\partial$ (see Proposition 2.14 there). With the same idea as in \cite{LXYZ24}, using Theorem \ref{thm:L^2 inverse}, we can give the following approximation result:
\begin{thm}\label{thm:limit metric}
  Let $M$ be a Riemannian manifold without boundary satisfying condition $\clubsuit$, and let $E$ be a flat vector bundle over $M$.
 Suppose $\{h_j\}_{j=1}^\infty $ is an increasing sequence of Riemannian metrics of class $C^2$ on $E$ such that $\mathfrak{Ric}_p+\Theta^{(E,h_j)}\geq 0$,
 and let $h:=\lim_{j\rw \infty}h_j$. Then $(M,E,h)$ satisfies the $p$-optimal $L^2$-estimate condition.
In particular, if $h$ is of class $C^2$, then $\mathfrak{Ric}_p+\Theta^{(E,h)}\geq 0$.
\end{thm}
\indent With Theorem \ref{thm:L^2 inverse} in hand, we can generalize the classical Pr\'ekopa's theorem in convex analysis.
\begin{thm}\label{thm:prekopa theorem}
  Let $M$ be a Riemannian manifold without boundary  satisfying condition $\clubsuit$, 
  $N$ be a $p$-convex Riemannian manifold of dimension $\geq p$,
and let $(E,\tilde{h})$ be a flat vector bundle over 
 $M\times N$ such that $\mathfrak{Ric}_p^{M\times N}+\Theta^{(E,\tilde{h})}\geq 0.$ 
 Define a metric $h$ on $E$ via   
 $$h_x:=\int_{N}\tilde{h}_{(x,y)} dV(y),\ \forall x\in M,$$
 then $(M,E,h)$ satisfies the $p$-optimal $L^2$-estimate condition. In particular, if $h$ is of class $C^2$,
 then $\mathfrak{Ric}_p^M+\Theta^{(E,h)}\geq 0$. 
\end{thm}
\indent When $M$ and $N$ are Euclidean spaces and $p=1$, Theorem \ref{thm:prekopa theorem}
is given in \cite{R13}. When $M$ and $N$ are convex domains of Euclidean spaces and $p=1$, then Theorem \ref{thm:prekopa theorem}
is a special case of \cite[Theorem 1.4]{DHJ23}. It is also possible to state and prove a version of Theorem \ref{thm:prekopa theorem} 
with group action as in \cite{DHJ23}, but we don't pursue this here. The proof of Theorem \ref{thm:prekopa theorem} is inspired by the proof 
of \cite[Theorem 1.6]{DZ21}. 
\ 
\subsection*{Acknowledgements}
    		The authors are grateful to Professor Fusheng Deng, their Ph.D. advisor, for valuable discussions on related topics.
\ \\
\section{Notations and conventions}\label{sec:notations}

\par In this section, we fix some notations and conventions that are needed in our discussions.\\
\indent Our convention for $\mathbb{N}$ is $\mathbb{N}:=\{0,1,2,\cdots\}$. Let $T$ be a symmetric linear transformation on an inner space $(V,\langle\cdot,\cdot\rangle)$, then we write $T\geq 0$ (resp. $T>0$) if 
$\langle Tu,u\rangle\geq 0$ (resp. $\langle Tu,u\rangle>0$) for all $u\in  V\setminus\{0\}.$\\ 
\indent When we say $M$ is a Riemannian manifold, we mean it is smooth, oriented, and which is equipped with a smooth Riemannian metric.
Its dimension is denoted by $\dim(M)$.  
The Lebesgue volume form on $M$ is denoted by $dV$, and the geodesic distance function on $M$ is denoted by 
$d(\cdot,\cdot)$. For any $x\in M$, and $r>0$, we set 
$$B(x,r):=\{y\in M|\ d(x,y)<r\}.$$ 
Let $A$ be a subset of $M$, then we write $\overline A$ for the closure of $A$ in $M$, and 
write $A\subset\subset M$ if $\overline A$ is a compact subset of $M$. If $A$ is a connected open set,
then we say $A$ is a domain of $M$.
Moreover, we use $|A|$ to denote the Lebesgue measure of $A$, and use $\diam(A)$ to denote the diameter of $A$.
When we say $(E,h)$ is a vector bundle over $M$, we mean $h$ is a $C^2$ Riemannian metric on $E$, and the metric on each 
fiber of $E$ is denoted by $\langle\cdot,\cdot\rangle$. If $(E_1,h_1),(E_2,h_2)$ are two vector bundles over $M$, 
then there is of course a vector bundle $(E_1\otimes E_2,h_1\otimes h_2)$ over $M$, where $h_1\otimes h_2$ is the tensor product of
metrics $h_1$ and $h_2$.\\
\indent Let $M$ be a Riemannian manifold without boundary, we use $TM:=\coprod_{x\in M}T_xM$ to denote the tangent bundle of $M$,
and $T^*M:=\coprod_{x\in M}T_x^*M$ to denote the cotangent bundle. The bundle of smooth $p$-forms on $M$ 
is denoted by $\Lambda^pT^*M.$ We always use $\nabla$ to denote the Levi-Civita connection on $M$,
 and we use $R$ to represent the Riemannian curvature tensor, i.e. 
 $$R(X,Y):=\nabla_X\nabla_Y-\nabla_Y\nabla_X-\nabla_{[X,Y]},$$
 where $X,Y$ are vector fields on $M$ of class $C^1$, and $[\cdot,\cdot]$ is the Lie bracket.
 Let $E$ be a smooth vector bundle over $M$, $k\in\mathbb{N}\cup\{\infty\}$,
 and let $U\subset\subset M$ be an open subset, then we use $C^k(U,E)$ denote the space of all sections of $E$ on $U$ which
 are of class $C^k$, and we use $C^k(U,E)$ (resp. $C_c^k(U,E)$) denote the space of all sections (resp. all compactly supported sections) of $E$ on an open neighborhood of $\overline{U}$ which
 are of class $C^k$.  Similarly, we use $L^2(U,E)$ (resp. $L_{\loc}^2(U,E)$) to denote of all Lebesgue measurable sections of $E$ on $U$ which
 are $L^2$-integrable (resp. locally $L^2$-integrable) on $U$ when $E$ has a Riemannian metric. When $E:=M\times \mc$, we write 
 $C^k(U):=C^k(U,E)$, and similarly for $C^k(\overline{U}).$ For any $f\in C^1(U)$, we use $\nabla f$ to denote the gradient of 
 $f$.\\
 \indent Let $(M,g)$ be a Riemannian manifold without boundary, then we use $d$ to represent the usual de Rham operator on $M$
  which acts on differential forms.
    In local coordinate, we may write 
    $$g=\sum_{i,j=1}^n g_{ij}dx_i\otimes dx_j,$$
    then we  use $(g^{ij})_{1\leq i,j\leq n}$ to represent the inverse matrix of $(g_{ij})_{1\leq i,j\leq n}$.
    For any $k\in\mathbb{N}$, let $\phi,\psi$ be two complex valued measurable $k$-forms on $M$, which have local expressions
    $$\phi=\frac{1}{k!}\sum_{1\leq i_1,\cdots,i_k\leq n}\phi_{i_1\cdots i_k}dx_{i_1}\wedge\cdots\wedge dx_{i_k},$$
    $$\psi=\frac{1}{k!}\sum_{1\leq j_1,\cdots,j_k\leq n}\psi_{j_1\cdots j_k}dx_{j_1}\wedge\cdots\wedge dx_{j_k},$$
    then we define 
   $$\langle \phi,\psi\rangle:=\frac{1}{k!}g^{i_1j_1}\cdots g^{i_k j_k}\phi_{i_1\cdots i_k}\overline{\psi_{j_1\cdots j_k}},$$
   and we also set 
   $$|\phi|^2:=\langle \phi,\phi\rangle.$$
   This definition is independent of the choice of local coordinates.\\
 \indent Let $M$ be a Riemannian manifold without boundary.
Let $X_1,\cdots,X_n$ be a local orthonormal frame of $TM$, and let $X_1^*,\cdots,X_n^*$ be the dual frame of $X_1,\cdots,X_n$. 
For any $\varphi\in C^2(M)$, set
$$\nabla^2\varphi:=\sum_{i,j=1}^n\varphi_{ij}X_i^*\otimes X_j^*,$$
and define
$$F_\varphi:=\sum_{i,j=1}^n \varphi_{ij}X_i^*\wedge X_j\lrcorner,$$
where $\lrcorner$ is the interior product. \\
\indent Let $M$ be a Riemannian manifold without boundary. Let $\Delta_p$ be the Hodge-Laplacian which acts 
 on $C^2(M,\Lambda^pT^*M)$, and let $-\Delta_0$ be the usual Laplace-Beltrami operator which acts on $C^2(M)$. For any $f\in C^2(M,\Lambda^pT^*M)$, the Weitzenb\"ock curvature operator $\mathfrak{Ric}_p$ on $f$ is defined by 
$$(\mathfrak{Ric}_pf)(Y_1,\cdots,Y_k):=\sum_{i=1}^k\sum_{j=1}^n (R(X_j,Y_i)f)(Y_1,\cdots,Y_{i-1},X_j,Y_{i+1},\cdots,Y_k),$$
where $X_1,\cdots,X_n$ is a local orthonormal frame of $TM$, $Y_1,\cdots,Y_k$ are vector fields on $M$ of class $C^1.$ 
When there are two more manifolds, we write $\mathfrak{Ric}_p^M$ for the Weitzenb\"ock curvature operator
for clarity. For any differential operator $P$ on $M$, we always use $P^*$ to denote the formal adjoint of $P$, and 
 use $\operatorname{Dom}(P)$ to denote the domain of $P$.
By the Weitzenb\"ock formula, we have 
$$\Delta_p=\nabla^*\nabla+\mathfrak{Ric}_p.$$
\indent Now let us recall the definitions of $p$-plurisubharmonic functions and $p$-convex manifolds 
 (see \cite{JLY14}). Suppose $M$ is a Riemannian manifold without boundary, then a real valued function $f\in C^2(M)$ is (strictly) $p$-plurisubharmonic 
  if the sum of any $p$-eigenvalues of $\nabla^2f$ is (strictly) greater than $0$.  
  $M$ is called (strictly) $p$-convex if  it has a
  (strictly) $p$-plurisubharmonic exhaustion function $\rho\in C^2(M).$ 
  Let $\Omega\subset\subset M$ be an open subset with a $C^k$-boundary defining function $\rho$ for some 
  $k\geq 2$, then $\partial\Omega$ is called  $p$-convex (resp. strictly $p$-convex) if for any $x\in\partial\Omega$, and any 
  $u\in \Lambda^p T_x^*M\setminus\{0\}$ such that $\nabla\rho\lrcorner u=0$, we have $\langle F_{\rho}u,u\rangle\geq 0$ (resp. $>0$). 
  Then we say $\rho$ is a $C^k$ $p$-convex (resp. strictly $p$-convex) boundary defining function of $\Omega$.
\section{The proof of Theorem \ref{thm:flat vector bundle}}\label{sec:flat}

\par To prove Theorem \ref{thm:flat vector bundle}, let us firstly introduce some notations.\\
\indent Let $(E,h)$ be a flat vector bundle of rank $r$ over $M$, which means that we can find a covering $\{U_\alpha\}_{\alpha\in I}$ of $M$ and a local frame 
$e_{\alpha,1},\cdots, e_{\alpha,r}$ of $E$ on $U_\alpha$ for each $\alpha$ such that the transition function between $e_{\alpha,1},\cdots,e_{\alpha,r}$ and $e_{\beta,1},\cdots,e_{\beta,r}$ are locally constant if $U_\alpha\cap U_\beta\neq \emptyset$.
For any $f\in C^1(M,\Lambda^pT^*M\otimes E)$ and any $\alpha\in I$, write 
$$f|_{U_\alpha}=\sum_{i=1}^r s^i\otimes e_{\alpha,i},$$
then we define 
$$df|_{U_\alpha}:=\sum_{i=1}^r ds^i\otimes e_{\alpha,i},\ D f|_{U_\alpha}:=\sum_{i=1}^r \nabla s^i\otimes e_{\alpha,i}.$$
Clearly, $df$ and $D f$ are globally defined. 
For any $f\in L^2_{\loc}(M,\Lambda^pT^*M\otimes E)$, write 
$$f|_{U_\alpha}=\sum_{i=1}^r s^i\otimes e_{\alpha,i},$$
and define  
$$\mathfrak{Ric}_pf|_{U_\alpha}:=\sum_{i=1}^r \mathfrak{Ric}_ps^i\otimes e_{\alpha,i},$$
then we know $\mathfrak{Ric}_p$ is a well defined operator. Furthermore, $\langle \mathfrak{Ric}_pf,f\rangle$ only depends on the curvature of $M$, the metric $h$ and $f$. 
Define 
$$
\Theta^{(E,h)}f|_{U_\alpha}:=-\sum_{i,j,k=1}^rdh_{\alpha}^{ij}\wedge\left(\nabla h_{\alpha,jk}\lrcorner s^k\right)\otimes e_{\alpha,i} -\sum_{i,j,k=1}^rh_{\alpha}^{ij} F_{h_{\alpha,jk}}s^k\otimes e_{\alpha,i},
$$
then $\Theta^{(E,h)}f$ is globally defined, and $\Theta^{(E,h)}f$ only depends on the metric $h$ and $f$.
Clearly, $\mathfrak{Ric}_p$ and $\Theta^{(E,h)}$ are symmetric in the following sense
$$\langle \mathfrak{Ric}_pf,g\rangle=\langle f,\mathfrak{Ric}_pg\rangle,\ \langle \Theta^{(E,h)}f,g\rangle=\langle f,\Theta^{(E,h)}g\rangle,$$
where $f,g\in L^2_{\loc}(M,\Lambda^pT^*M\otimes E)$. Use \cite[Lemma 7.1]{JLY14},
then we have  
$$ d(h_\alpha^{jk}\nabla h_{\alpha,ij}\lrcorner s^i)\otimes e_{\alpha,k}
 +h_\alpha^{jk}\nabla h_{ij}\lrcorner ds^i\otimes e_{\alpha,k}=
 h_{\alpha}^{jk}\nabla_{\nabla h_{\alpha,ij}}s^i\otimes e_{\alpha,k}+\Theta^{(E,h)}f,$$
so one can easily prove that 
\begin{equation}\label{metric}
\Theta^{(E,he^{-\varphi})}=\Theta^{(E,h)}+F_\varphi,\ \forall \varphi\in C^2(M).
\end{equation}
\indent Suppose $T$ is a symmetric linear transformation on $L^2_{\loc}(M,\Lambda^pT^*M\otimes E)$ such that 
for almost every $x\in M$,  
and  there exists $\alpha\in[0,+\infty)$ such that 
$$\langle Tu,u\rangle\geq 0,\ |\langle f(x),u\rangle|^2\leq \alpha \langle Tu,u\rangle,
\ \forall u\in (\Lambda^pT^*M\otimes E)_x,$$
then we define $\langle T^{-1}f(x),f(x)\rangle$ to be the minimal such number $\alpha$. This definition is given by 
Demailly in Page 371 of  \cite{D12}.\\  
\indent The following lemma plays a fundamental role in our proof of Theorem \ref{thm:flat vector bundle}.
\begin{lem}\label{lem:bochner}
Let $M$ be a Riemannian manifold without boundary of dimension $n$, $(E,h)$ be a flat vector bundle over $M$, and let $\Omega\subset\subset M$ be an open subset with a $C^2$-boundary defining function $\rho$.
 Then for any $f\in C^1(\overline\Omega,\Lambda^pT^*M\otimes E)\cap\operatorname{Dom}(d^*),$
we have 
\begin{align*}
\int_{\Omega}|df|^2dV+\int_{\Omega}|d^*f|^2dV&=\int_{\Omega}|D f|^2dV+\int_{\Omega}\left\langle 
\left(\mathfrak{Ric}_p+\Theta^{(E,h)}\right)f,f\right\rangle dV\\
&\quad +\int_{\partial \Omega}\langle F_\rho f,f\rangle\frac{dS}{|\nabla\rho|}.
\end{align*}
\end{lem}
\begin{proof}
Fix  $f\in C^1(\overline\Omega,\Lambda^pT^*M\otimes E)\cap\operatorname{Dom}(d^*),$ then we may assume 
 $f\in C^2(\overline\Omega,\Lambda^pT^*M\otimes E)$ by approximation. 
Set 
$$\square f:=dd^*f+d^*df.$$
We have 
$$d^*f|_{U_\alpha}=h_\alpha^{jk}d^*(h_{\alpha,ij}s^i)\otimes e_{\alpha,k},$$
then use the Stokes theorem and the fact $f\in \operatorname{Dom}(d^*)$, we have 
\begin{equation}\label{eqqu:1}
\int_{\Omega}|df|^2dV+\int_{\Omega}|d^*f|^2dV=\int_{\Omega}\langle \square f,f\rangle dV+\int_{\partial\Omega}\langle f, \nabla\rho\lrcorner df\rangle \frac{dS}{|\nabla\rho|}.
\end{equation}
Note that 
$$d^*f|_{U_\alpha}=d^*s^k\otimes e_{\alpha,k}-h_\alpha^{jk}\nabla h_{\alpha,ij}\lrcorner s^i\otimes e_{\alpha,k},$$
then use \cite[Lemma 7.1]{JLY14}, we get 
\begin{align*}
 &\quad \square f|_{U_\alpha}\\
 &=\Delta_ps^k\otimes e_{\alpha,k}-d(h_\alpha^{jk}\nabla h_{\alpha,ij}\lrcorner s^i)\otimes e_{\alpha,k}
 -h_\alpha^{jk}\nabla h_{ij}\lrcorner ds^i\otimes e_{\alpha,k}\\
 &=\Delta_ps^k\otimes e_{\alpha,k}
 -h_{\alpha}^{jk}\nabla_{\nabla h_{\alpha,ij}}s^i\otimes e_{\alpha,k}-\Theta^{(E,h)}f,
\end{align*}
so we know 
\begin{equation}\label{eqqu:2}
\langle \square f,f\rangle|_{U_\alpha}=\langle \Delta_ps^i,s^j\rangle h_{\alpha,ij}
-\langle \nabla_{\nabla h_{\alpha,ij}}s^i,s^j\rangle -\langle \Theta^{(E,h)}f,f\rangle.
\end{equation}
Choose a local orthonormal frame $X_1,\cdots,X_n$ near a point $x\in U_\alpha$.
Define a vector field  
$$Z|_{U_\alpha}:=\sum_{\lambda=1}^n\langle \nabla_{X_\lambda}s^i,s^j\rangle h_{\alpha,ij}X_\lambda,$$
then $Z$ is globally defined. 
Use $\operatorname{div}(Z)$ to denote the divergence of $Z$. 
By the Weitzenb\"ock formula, we know 
\begin{align}\label{eqqu:3}
&\quad \langle \Delta_ps^i,s^j\rangle h_{\alpha,ij}-\langle \nabla_{\nabla h_{\alpha,ij}}s^i,s^j\rangle\\
&=-\sum_{\lambda=1}^n\langle\nabla_{X_\lambda}\nabla_{X_\lambda}s^i,s^j\rangle
+\sum_{\mu=1}^n\langle \nabla_{\nabla_{X_\mu}X_\mu}s^i,s^j\rangle h_{\alpha,ij}-\langle \nabla_{\nabla h_{\alpha,ij}}s^i,s^j\rangle
\nonumber\\
&\quad +\langle\mathfrak{Ric}_ps^i,s^j\rangle h_{\alpha,ij}\nonumber\\
&=-\sum_{\lambda=1}^n\langle\nabla_{X_\lambda}\nabla_{X_\lambda}s^i,s^j\rangle
+\sum_{\lambda,\mu=1}^n\langle \nabla_{X_\mu}X_\mu,X_\lambda\rangle \langle \nabla_{X_\lambda}s^i,s^j\rangle h_{\alpha,ij}\nonumber\\
&\quad -\langle \nabla_{\nabla h_{\alpha,ij}}s^i,s^j\rangle+\langle\mathfrak{Ric}_ps^i,s^j\rangle h_{\alpha,ij}\nonumber\\
&=-\sum_{\lambda=1}^nX_\lambda(\langle \nabla_{X_\lambda}s^i,s^j\rangle h_{\alpha,ij})
-\sum_{\lambda,\mu=1}^n\langle \nabla_{X_\mu}X_\lambda,X_\mu\rangle \langle \nabla_{X_\lambda}s^i,s^j\rangle h_{\alpha,ij}
\nonumber\\
&\quad +\sum_{\lambda=1}^n\langle \nabla_{X_\lambda}s^i,\nabla_{X_\lambda}s^j\rangle h_{\alpha,ij}
+\langle \mathfrak{Ric}_pf,f\rangle\nonumber\\
&=-\operatorname{div}(Z)+|Df|^2+\langle \mathfrak{Ric}_pf,f\rangle.\nonumber
\end{align}
By the divergence theorem, we get 
\begin{equation}\label{eqqu:4}
\int_{\Omega}\operatorname{div}(Z)dV=\int_{\partial\Omega}\langle Z,\nabla\rho\rangle\frac{dS}{|\nabla\rho|}=
\int_{\partial\Omega}\langle D_{\nabla\rho}f,f\rangle \frac{dS}{|\nabla\rho|}.
\end{equation}
Combing Equalities (\ref{eqqu:1}), (\ref{eqqu:2}), (\ref{eqqu:3}), and (\ref{eqqu:4}), it suffices to prove that 
 \begin{equation}\label{eqqu:5}
 \int_{\partial\Omega}\langle f,\nabla\rho\lrcorner df-D_{\nabla\rho}f\rangle \frac{dS}{|\nabla\rho|}=\int_{\partial\Omega}\langle F_\rho f,f\rangle \frac{dS}{|\nabla\rho|}.
 \end{equation}
 \indent Since $f\in \operatorname{Dom}(d^*)$, then we know 
  $$\nabla\rho\lrcorner f=\rho h$$
  for some differential form $h$ of class $C^1$ on $\overline\Omega$, which implies that 
  \begin{equation}\label{eqqu:6}
  \int_{\partial\Omega}\langle f,d(\nabla\rho\lrcorner f)\rangle \frac{dS}{|\nabla\rho|}=0.
  \end{equation}
  Similar proof with the proof of \cite[Lemma 7.1]{JLY14} implies that 
  \begin{equation}\label{eqqu:7}
  d(\nabla\rho\lrcorner f)+\nabla\rho\lrcorner (df)=D_{\nabla\rho}f+F_\rho f.
  \end{equation}
  Combining Equalities (\ref{eqqu:6}) and (\ref{eqqu:7}), we know Equality (\ref{eqqu:5}) holds.
\end{proof}
   With the same proof as that of Lemma \ref{lem:bochner}, we also have the following: 
\begin{lem}\label{lem:bochner without boundary}
Let $M$ be a Riemannian manifold without boundary, and let $(E,h)$ be a flat vector bundle over $M$.
 Then for any $f\in C^1_c(M,\Lambda^pT^*M\otimes E),$ we have 
$$
\int_{M}|df|^2dV+\int_{M}|d^*f|^2dV=\int_{M}|D f|^2dV+\int_{M}\left\langle 
\left(\mathfrak{Ric}_p+\Theta^{(E,h)}\right)f,f\right\rangle dV.
$$
\end{lem}
  Through a minor modification of the proof of \cite[Proposition 2.1.1]{H65},
  one has the following density lemma.
    \begin{lem}\label{lem:approximation}
     $C^1(\overline{\Omega},\Lambda^pT^*M\otimes E)\cap \dom(d^*)$ is dense in $\dom(d)\cap \dom(d^*)$ under the graph norm
     given by 
     $$f\mapsto \left(\int_{\Omega}|f|^2dV\right)^{1/2}+\left(\int_{\Omega}|df|^2dV\right)^{1/2}+\left(\int_{\Omega}|d^*f|^2dV\right)^{1/2}.$$
    \end{lem}
  
  \indent Now we can give the proof of Theorem \ref{thm:flat vector bundle}. For convenience, we restate it here.
 \begin{thm}[= Theorem \ref{thm:flat vector bundle}]
Let $M$ be an  $n$-dimensional $p$-convex Riemannian manifold without boundary for some  $p\in\{1,\cdots,n\}$, and let 
$(E,h)$ be a flat vector bundle over $M$. Suppose $\mathfrak{Ric}_p+\Theta^{(E,h)}\geq 0,$ then for any 
$d$-closed $f\in L^2_{\loc}(M,\Lambda^pT^*M\otimes E)$ satisfying 
$$\int_{M}\left\langle \left(\mathfrak{Ric}_p+\Theta^{(E,h)}\right)^{-1}f,f\right\rangle dV<\infty,$$
there exists $u\in L^2(M,\Lambda^pT^*M\otimes E)$ such that 
$$du=f\text{ and }\int_{M}|u|^2dV\leq \int_{M}\left\langle \left(\mathfrak{Ric}_p+\Theta^{(E,h)}\right)^{-1}f,f\right\rangle dV.$$
Moreover, if $f$ is smooth, then $u$ can be taken to be smooth.
\end{thm}
\begin{proof}
We firstly prove the first part. Thanks to Sard's theorem, we know $M$ can be exhausted by relatively compact open subset with $C^2$-boundary
 which is $p$-convex. By taking weak limit and approximation, we only need to work on a open subset $\Omega\subset\subset M$ 
 with $C^2$-boundary which is $p$-convex. By Lemma  \ref{lem:bochner} and \ref{lem:approximation}, 
 for all $g\in \operatorname{Dom}(d)\cap \operatorname{Dom}(d^*),$ we know 
 $$\int_{\Omega}|dg|^2dV+\int_{\Omega}|d^*g|^2dV\geq \int_{\Omega}\left\langle 
\left(\mathfrak{Ric}_p+\Theta^{(E,h)}\right)g,g\right\rangle dV.$$
Similar to the proof of \cite[Theorem 4.5, Chapter VIII]{D12}, the conclusion follows from 
the above inequality by the Hahn-Banach extension theorem and Riesz representation theorem.\\
\indent By taking the minimal solution, the regularity of $u$ follows from the fact $dd^*+d^*d$ is an elliptic operator.
\end{proof}
\section{The proof of Theorem \ref{thm:improved $L^2$-estimates}}

 \par To prove Theorem \ref{thm:improved $L^2$-estimates}, we need the following important lemma.
 \begin{lem}\label{lem:differential form}
 Let $M$ be a Riemannian manifold without boundary, $\Omega\subset\subset M$ be an open subset with 
 smooth boundary such that $\mathfrak{Ric}_p\geq 0 $ on $\Omega$. Then there is a constant 
 $\delta:=\delta(\Omega)>0$ such that 
 $$\delta\int_{\Omega}|f|^2dV\leq \int_{\Omega}|\nabla f|^2dV+\int_{\partial\Omega}|f|^2dS,\ \forall f\in C^1(\overline\Omega,\Lambda^pT^*M).$$
 Moreover, if $M=\mr^n$ ($n\geq 3$), then $\delta$ can be chosen to be a constant only depends on $n,|\Omega|$.
 \end{lem}
 \begin{proof}
  The first inequality follows from Corollary 1.5 of \cite{DHQ25}. 
  The case $M=\mr^n$ follows from the following Sobolev-type inequality: there is a constant $\delta:=\delta(n,|\Omega|)>0$
  such that 
  $$\delta\int_{\Omega}|f|^2dV\leq \int_{\Omega}|\nabla f|^2dV+\int_{\partial\Omega}|f|^2dS,\ \forall f\in C^1(\overline\Omega),$$
  which can be proved by using Inequality (2.1) of \cite{CM16} and the H\"older inequality.
 \end{proof}
  Use Lemma \ref{lem:differential form}, one has the following: 
  \begin{cor}\label{cor:improved L^2 estimate}
  With the same assumptions and notations as in Theorem \ref{thm:improved $L^2$-estimates}.
  Choose  a constant $\lambda>0$ such that 
   $$\frac{1}{\lambda} I\leq h|_{\Omega}\leq \lambda I,$$
   where $I$ is the identity matrix.
   Then there is constant 
 $\delta:=\delta(\Omega)>0$ such that 
 $$\frac{\delta}{\lambda^2}\int_{\Omega}|f|^2dV\leq \int_{\Omega}|D f|^2dV+\int_{\partial\Omega}|f|^2dS,\ \forall f\in C^1(\overline\Omega,\Lambda^pT^*M\otimes E).$$
 Moreover, if $M=\mr^n$ ($n\geq 3$), then $\delta$ can be chosen to be a constant only depends on $n,|\Omega|$.
  \end{cor}
  \begin{proof}
    Choose a frame $e_1,\cdots,e_r$ of $E$, where $r$ is the rank of $E$.
   Fix $f\in C^1(\overline\Omega,\Lambda^pT^*M\otimes E)$, and write 
   $$f=\sum_{i=1}^r s^i\otimes e_{i}.$$
   By Lemma \ref{lem:differential form}, there is a constant 
    $\delta:=\delta(\Omega)>0$ such that for any $i$, we have 
  $$\delta\int_{\Omega}| s^i|^2dV\leq \int_{\Omega}|\nabla s^i|^2dV+\int_{\partial\Omega}|s^i|^2dS.$$
  Moreover, if $M=\mr^n$ ($n\geq 3$), then $\delta$ can be chosen to be a constant only depends on $n,|\Omega|$.
  Then we get
  \begin{align*}
   \int_{\Omega}|f|^2dV&\leq\lambda\sum_{i=1}^r\int_{\Omega}|s_i|^2dV\leq \frac{\lambda}{\delta}\sum_i \left(\int_{\Omega}|\nabla s^i|^2dV+\int_{\partial\Omega}|s^i|^2dS\right)\\
   &\leq \frac{\lambda^2}{\delta}\left(\int_{\Omega}|D f|^2dV+\int_{\partial\Omega}|f|^2dS\right).
  \end{align*}
  \end{proof}
  Now we give the proof of Theorem \ref{thm:improved $L^2$-estimates}
  \begin{thm}[= Theorem \ref{thm:improved $L^2$-estimates}]
    Let $M$ be a Riemannian manifold without boundary, $\Omega\subset\subset M$ be an open subset with 
    a smooth strictly $p$-convex boundary defining function $\rho$
    such that $\mathfrak{Ric}_p\geq 0$ on $\Omega$, and let $(E,h)$ be a trivial vector bundle over $M$
    such that $\Theta^{(E,h)}\geq 0$ on $\Omega$. Then 
    there is a constant $\delta:=\delta(\Omega,\rho,h)>0$ such that for any nonzero 
    $d$-closed $f\in L^2(\Omega,\Lambda^pT^*M\otimes E)$ satisfying 
    $$N_f:=\int_{\Omega}\left\langle \left(\mathfrak{Ric}_p+\Theta^{(E,h)}\right)^{-1}f,f\right\rangle dV<\infty,$$
there exists $u\in L^2(\Omega,\Lambda^pT^*M\otimes E)$ such that $du=f$ and 
$$\int_{\Omega}|u|^2dV\leq \frac{\|f\|_{L^2(\Omega)}}{\sqrt{\|f\|^2_{L^2(\Omega)}+\delta N_f}}\int_{\Omega}\left\langle \left(\mathfrak{Ric}_p+\Theta^{(E,h)}\right)^{-1}f,f\right\rangle dV,$$
    where $\|f\|_{L^2(\Omega)}^2:=\int_{\Omega}|f|^2dV$. Moreover, if $f$ is smooth, then $u$ can be taken to be smooth.
\end{thm}
\begin{proof}
  We only need to prove the first part.
  Choose a constant $\lambda_1>0$ such that 
  $$\frac{1}{\lambda_1} I\leq h|_{\Omega}\leq \lambda_1 I.$$
  Choose a constant $\lambda_2>0$ such that for any $x\in\partial\Omega$, and any 
  $u\in \Lambda^p T_x^*M\otimes E_x$ satisfying $\nabla\rho\lrcorner u=0$, we have 
  $$\langle F_{\rho}u,u\rangle\geq \lambda_2|u|^2.$$ 
  Use Lemma  \ref{lem:bochner}, Lemma \ref{lem:approximation}, and Corollary \ref{cor:improved L^2 estimate},
  there is a constant $\delta:=\delta(\Omega)>0$ such that for all 
  $g\in \operatorname{Dom}(d)\cap \operatorname{Dom}(d^*),$ we know 
 $$\int_{\Omega}|dg|^2dV+\int_{\Omega}|d^*g|^2dV\geq \int_{\Omega}\left\langle 
\left(\mathfrak{Ric}_p+\Theta^{(E,h)}\right)g,g\right\rangle dV+\lambda\int_{\Omega}|g|^2dV,$$
 where 
 $$\lambda:=\frac{\delta\min\{\lambda_2,1\}}{\lambda_1^2}.$$
 Moreover, if $\Omega=\mr^n$ ($n\geq 3$), then $\delta$ can be chosen to be a constant only depends on $n,|\Omega|$.
 By definition of $\left(\Theta^{(E,h)}\right)^{-1}$ and the Cauchy-Schwarz inequality,  for all 
  $g\in \operatorname{Dom}(d)\cap \operatorname{Dom}(d^*),$ we get
 \begin{align*}
 &\quad \left|\int_{\Omega}\langle f,g\rangle dV\right|^2+\lambda\frac{\left|\int_{\Omega}\langle f,g\rangle dV\right|^2}{\|f\|^2} N_f\\
 &\leq \int_{\Omega}\left\langle\left(\mathfrak{Ric}_p+ \Theta^{(E,h)}\right)g,g\right\rangle dV\cdot N_f+\lambda\int_{\Omega}|g|^2dV\cdot N_f\\
 &\leq \left(\int_{\Omega}|dg|^2dV+\int_{\Omega}|d^*g|^2dV\right)N_f.
 \end{align*}
 The remaining is completed by the Hahn-Banach extension theorem and Riesz representation theorem.
\end{proof}
 \section{The proof of Theorem \ref{thm:L^2 inverse}}

 \par To prove Theorem \ref{thm:L^2 inverse}, we need the following useful lemma:
 \begin{lem}\label{lem:extension}
  Let $(M,g)$ be a Riemannian manifold without boundary of dimension $n\geq 2$, which admits a strictly $p$-plurisubharmonic function $\eta\in C^2(M)$,
  where $p$ satisfies 
   $$\begin{cases}
     p=n, & \mbox{if } n=2, \\
     3\leq p\leq n, & \mbox{if }n\geq 3,
   \end{cases}$$
  and let $U$ be an open neighborhood of  a point $x\in M$. Then for any strictly $p$-plurisubharmonic function $\varphi\in C^2(U)$, there is an open neighborhood 
  $V\subset U$ of $x$ and a strictly $p$-plurisubharmonic function $\psi\in C^2(M)$ such that $\psi|_{V}=\varphi|_{V}$.
 \end{lem}
 \begin{proof}
 Let $r_0$ be the injective radius of $M$ at $x$.
 Choose $r_1>0$ such that $r_1<r_0$, and fix a normal coordinate $(B(x,r_1);y_1,\cdots,y_n)$ near $x$,
 then we know $g(x)$ is the identity matrix.
 Now choose $R>0$ such that $5R<r_1$
 and $B(x,5R)\subset\subset U$. \\
 {\bf Claim:} Shrinking $R$ if necessary, we may find a function $f\in C^\infty(M\setminus\{x\})$
 such that the following holds:
 \begin{itemize}
 \item[(i)] $f$ is $p$-plurisubharmonic  on $B(x,2R)\setminus\{x\}.$
 \item[(ii)] $f(x)=-\infty$, and $f$ is continuous at $x.$
 \item[(iii)] $f(y)=0$ for all $y\in\partial B(x,4R)$.
 \end{itemize}
 
 \indent Indeed, when $n=2$, let $f$ be the (negative) Dirichlet Green function of $B(x,4R)$. 
 When $n\geq 3$, set 
 $$h(y):=\ln \left(\frac{ d(x,y)^2}{16R^2}\right),\ \forall y\in M,$$
 then by the Hessian comparison theorem (see \cite[Theorem 11.7]{L18}), $h$ is strictly $p$-plurisubharmonic on $B(x,2R)\setminus\{x\}$ by shrinking 
 $R$ if necessary as $p\geq 3$. Choose a bump function $\chi\in C_c^\infty(M)$ such that 
 $$\chi|_{B(x,2R)}\equiv 1,\ \chi(y)=0,\ \forall y\in M\setminus B(x,5R).$$
 Set $f:=\chi h$, then $f$ satisfies the conditions of the {\bf Claim}.\\
 \indent Define 
 $$\rho:=\eta+\frac{f}{m}+m,$$
 where $m>>1$ enough large is an integer such that the following holds:
  \begin{itemize}
    \item[(i)] $\rho(y)\geq 4+\varphi(y)$ for all $y\in\partial B(x,2R)$.
    \item[(ii)] $\rho$ is strictly $p$-plurisubharmonic on $B(x,5R)\setminus B(x,R)$.
  \end{itemize}
  Clearly, $\rho\in C^2(M\setminus\{x\}),\ \rho(x)=-\infty,$  $\rho$ is continuous at $x$,
  and $\rho$ is strictly $p$-plurisubharmonic on $B(x,2R)\setminus\{x\}.$
 Define 
 $$\phi(y):=\begin{cases}
              \rho(y), & \mbox{ if }y\in M\setminus B(x,2R), \\
              \max_{(1,1)}\{\varphi(y),\rho(y)\}, & \mbox{ if } y\in B(x,2R), \
            \end{cases}$$
 where $\max_{(1,1)}$ is the regularized maximum function defined in \cite[Lemma 5.18, Chapter I]{D12}.
 Then we know $\phi\in C^2(M)$, $\phi$ is strictly $p$-plurisubharmonic on $B(x,5R),$ 
  $\phi=\varphi$ in $B(x,r_2)$ for some $0<r_2<R$ as $\rho$ is continuous at $x$. Set 
 $$\psi(y):=\begin{cases}
              \eta(y)+m, & \mbox{ if }y\in M\setminus B(x,4R), \\
              \phi(y), & \mbox{ if } y\in B(x,4R), \
            \end{cases}$$ 
 then $\psi\in C^2(M)$ is strictly $p$-plurisubharmonic on $M$ and $\psi=\varphi$ on $B(x,r_2)$.
 \end{proof}
 \begin{thm}[= Theorem \ref{thm:L^2 inverse}]
 Let $M$ be a  Riemannian manifold without boundary  satisfying condition  $\clubsuit$, and let $(E,h)$ be a flat vector bundle over $M$, which satisfies the following: for any strictly $p$-plurisubharmonic function $\psi\in C^2(M)$, and any $d$-closed $f\in  C^\infty_c(M,\Lambda^pT^*M\otimes E),$ 
there exists a Lebesgue measurable differential form $u$ such that 
$$du=f\text{ and }\int_{M}|u|^2e^{-\psi}dV\leq \int_{M}\langle F_\psi^{-1}f,f\rangle e^{-\psi}dV.$$
Then we have $\mathfrak{Ric}_p+\Theta^{(E,h)}\geq 0$.
\end{thm}
\begin{proof}
 Set $B^{(E,h)}:=\mathfrak{Ric}_p+\Theta^{(E,h)}$.
 We argue by contradiction, suppose the conclusion $B^{(E,h)}\geq 0$ not holds, then there is 
$x_0\in M$ and $\xi_0\in \Lambda^pT_{x_0}^*M\otimes E_{x_0}$ such that 
$|\xi_0|=1$ and $\langle B^{(E,h)}\xi_0,\xi_0\rangle=-2c$ for some $c>0.$
Choose a normal coordinate $(U;x_1,\cdots,x_n)$ near $x_0$.
Choose $R>0$ such that $B(x_0,2R)\subset\subset U$, and choose $\chi\in C_c^\infty(B(x_0,2R))$ such that 
$\chi|_{B(0,R)}=1.$ Note that for any $i_1<\cdots<i_p$, we may write 
$dx^{i_1}\wedge\cdots\wedge dx^{i_p}=d\eta$ for some $(p-1)$-form $\eta$,
then we may write $\xi_0=d\eta_0$ for some $(p-1)$-form $\eta_0$.
Set $f:=d(\chi \eta_0),$ then we know $f\in C_c^\infty(M,\Lambda^pT^*M\otimes E)$ is $d$-closed
and $f|_{B(x_0,R)}=\xi_0$, and then we get  
$$\langle B^{(E,h)}f,f\rangle<-c\text{ on }B(x_0,R).$$
Since the Hessian of $d(x,x_0)^2$ is positive definite at $x_0$,
then by Lemma \ref{lem:extension}, we can construct a strictly $p$-plurisubharmonic function 
$\psi\in C^2(M)$ such that 
$$\psi(x)=d(x,x_0)^2-R^2,\ \forall x\in B(x_0,2R),$$ 
where one may shrink $R$ if necessary.
For each $m\in\mathbb{N}\setminus\{0\}$, set $\psi_m:=m\psi$, and set $\alpha_m:=(F_{\psi_m})^{-1}f=\frac{1}{m}F_{\psi}^{-1}f$.
Similar as the proof of \cite[Theorem 1.1]{DNWZ23},
using Lemma \ref{lem:bochner without boundary}, the Cauchy-Schwarz inequality and the assumption, we get 
\begin{equation}\label{inequality}
\int_{M}\langle B^{(E,h)}\alpha_m,\alpha_m\rangle e^{-\psi_m}dV+\int_{M}|D\alpha_m|^2e^{-\psi_m}dV\geq 0,\ \forall m>0.
\end{equation}
Shrink $R$ if necessary, we may assume 
$$\sup_{B(x_0,R)}|D\alpha_m|\leq \frac{\sqrt{c}}{2m},\ \sup_{B(x_0,R)}\langle B^{(E,h)}\alpha_m,\alpha_m\rangle\leq -\frac{3c}{4m^2}.$$
Choose $C>0$ such that for all $m>0$, we have 
$$|D\alpha_m|\leq \frac{\sqrt{C}}{m},\ \langle B^{(E,h)}\alpha_m,\alpha_m\rangle\leq \frac{C}{m^2}.$$
Therefore, for any $m>0$, we get 
\begin{align*}
&\quad m^2\left(\int_{M}\langle B^{(E,h)}\alpha_m,\alpha_m\rangle dV+\int_{M}|D\alpha_m|^2dV\right)\\
&\leq -\frac{c}{2}\int_{B(x_0,R)}e^{-\psi_m}dV+2C\int_{B(x_0,2R)\setminus B(x_0,R)}e^{-\psi_m}dV\\
&\leq -\frac{c}{2}|B(x_0,R)|+2C\int_{B(x_0,2R)\setminus B(x_0,R)}e^{-\psi_m}dV.
\end{align*}
Use  Lebesgue dominated convergence theorem, we have 
$$m^2\left(\int_{M}\langle B^{(E,h)}\alpha_m,\alpha_m\rangle e^{-\psi_m}dV+\int_{M}|D\alpha_m|^2e^{-\psi_m}dV\right)<0$$
for enough large $m$, which is a contradiction with Inequality (\ref{inequality}).
\end{proof}
\section{The proofs of Theorem \ref{thm:limit metric} and \ref{thm:prekopa theorem}}

\par Firstly, we give the proof of  Theorem \ref{thm:limit metric}.
\begin{thm}[= Theorem \ref{thm:limit metric}]
 Let $M$ be a Riemannian manifold without boundary satisfying condition $\clubsuit$, and let $E$ be a flat vector bundle over $M$.
 Suppose $\{h_j\}_{j=1}^\infty $ is an increasing sequence of Riemannian metrics of class $C^2$ on $E$ such that $\mathfrak{Ric}_p+\Theta^{(E,h_j)}\geq 0$,
 and let $h:=\lim_{j\rw \infty}h_j$. Then $(M,E,h)$ satisfies the $p$-optimal $L^2$-estimate condition.
In particular, if $h$ is of class $C^2$, then $\mathfrak{Ric}_p+\Theta^{(E,h)}\geq 0$.
\end{thm}
\begin{proof}
 It suffices to prove $(M,E,h)$ satisfies the $p$-optimal $L^2$-estimate condition by Corollary \ref{cor:equivalent}.
 Fix a $d$-closed $f\in C_c^\infty(M,\Lambda^pT^*M\otimes E)$, and fix a strictly $p$-plurisubharmonic function $\psi\in C^2(M)$.
 By  Corollary \ref{cor:equivalent}, for any $j\geq 1$, we can find a Lebesgue measurable differential form $u_j$ such that 
 $du_j=f$, and
 $$\int_{M}h_j(u_j,u_j)e^{-\psi}dV\leq \int_{M}h_j(F_{\psi}^{-1}f,f)e^{-\psi}dV\leq\int_{M}h(F_{\psi}^{-1}f,f)e^{-\psi}dV,$$
 so we know 
 $$\int_{M}h_k(u_j,u_j)e^{-\psi}dV\leq \int_{M}h(F_{\psi}^{-1}f,f)e^{-\psi}dV,\ \forall j\geq k\geq 1.$$
 By taking a weakly convergent subsequence of $u_j$ and use the Cantor's diagonal argument, then we may assume $u_j$ converges weakly to a Lebesgue measurable differential form $u$ in the space $L_{\loc}^2(M,\Lambda^pT^*M\otimes E)$, and for any $k\geq 1$, we get 
 $$\int_{M}h_k(u,u)e^{-\psi}dV\leq \liminf_{j\rw \infty}\int_{M}h_k(u_j,u_j)e^{-\psi}dV\leq \int_{M}h(F_{\psi}^{-1}f,f)e^{-\psi}dV.$$
  Thus, we have  
 $du=f$, and use  Fatou's lemma, we conclude that
 $$\int_{M}h(u,u)e^{-\psi}dV\leq \liminf_{k\rw \infty}\int_{M}h_k(u,u)e^{-\psi}dV\leq\int_{M}h(F_{\psi}^{-1}f,f)e^{-\psi}dV.$$
 Hence, $(M,E,h)$ satisfies the $p$-optimal $L^2$-estimate condition.
\end{proof}
\indent Secondly, we give the proof of Theorem \ref{thm:prekopa theorem}.
\begin{thm}[= Theorem \ref{thm:prekopa theorem}]
  Let $M$ be a Riemannian manifold without boundary  satisfying condition $\clubsuit$, 
  $N$ be a $p$-convex Riemannian manifold of dimension $\geq p$,
and let $(E,\tilde{h})$ be a flat vector bundle over 
 $M\times N$ such that $\mathfrak{Ric}_p^{M\times N}+\Theta^{(E,\tilde{h})}\geq 0.$ 
 Define a metric $h$ on $E$ via   
 $$h_x:=\int_{N}\tilde{h}_{(x,y)} dV(y),\ \forall x\in M,$$
 then $(M,E,h)$ satisfies the $p$-optimal $L^2$-estimate condition. In particular, if $h$ is of class $C^2$,
 then $\mathfrak{Ric}_p^M+\Theta^{(E,h)}\geq 0$. 
\end{thm}
\begin{proof}
 We only need to prove $(M,E,h)$ satisfies the $p$-optimal $L^2$-estimate condition. 
 Clearly,  $M\times N$  is $p$-convex and  we may regard $(E,h)$ as a flat vector bundle over $M$.
 Fix a $d$-closed $f\in C_c^\infty(M,\Lambda^pT^*M\otimes E)$, and fix a strictly $p$-plurisubharmonic function $\psi\in C^2(M)$.
 By  the proof of Theorem \ref{thm:flat vector bundle} and use Equality (\ref{metric}),  we can find a Lebesgue measurable differential form $u$ such that 
 $du=f$, and
 $$\int_{M\times N}\tilde{h}(u,u)e^{-\psi}dV\leq \int_{M\times N}\tilde{h}(F_{\psi}^{-1}f,f)e^{-\psi}dV<\infty.$$
 Since $f$ is independent of $y\in N$, then we regard $u$ as a Lebesgue measurable differential form on $M$. 
 Using Fubini's theorem, we conclude that 
 $$\int_Mh(u,u)e^{-\psi}dV\leq \int_Mh(F_{\psi}^{-1}f,f) e^{-\psi}dV,$$
   so $(M,E,h)$ satisfies the $p$-optimal $L^2$-estimate condition.
\end{proof}
 \bibliographystyle{alphanumeric}
 
\end{document}